\documentclass[reqno]{amsart}
\usepackage{amsmath,amsthm}
\usepackage{amscd}
\usepackage{amsfonts,amssymb,latexsym}
\usepackage{graphicx} 
\usepackage{listings}

\usepackage{color}
\definecolor{darkred}{rgb}{1,0,0}
\definecolor{darkgreen}{rgb}{0,1,0}
\definecolor{darkblue}{rgb}{0,0,1}
\usepackage{hyperref}
\hypersetup{colorlinks,
linkcolor=darkblue,
filecolor=darkblue,
urlcolor=darkred,
citecolor=darkgreen}

\newtheorem{proposition}{Proposition}[section]
\newtheorem{theorem}[proposition]{Theorem}

\newtheorem{conjecture}[proposition]{Conjecture}

\newtheorem{remark}[proposition]{Remark}

\numberwithin{equation}{section}

\newcommand{\SH}{{\mathcal{H}}}

\newcommand{\SM}{{\mathcal{M}}}

\newcommand{\SP}{{\mathcal{P}}}

\newcommand{\LL}{\mathbb{L}}

\newcommand{\ZZ}{\mathbb{Z}}
\newcommand{\NN}{\mathbb{N}}
\newcommand{\CC}{\mathbb{C}}

\newcommand{\Sym}{\operatorname{Sym}}
\newcommand{\Alt}{\operatorname{Alt}}

\newcommand{\VarC}{\mathcal{V}ar_{\CC}}
\newcommand{\ChowC}{\mathcal{CM}_{\CC}}
\newcommand{\Jac}{\operatorname{Jac}}

\newcommand{\coeff}{\mathop{\mathrm{coeff}}}

\newcommand{\rk}{\operatorname{rk}}
\newcommand{\End}{\operatorname{End}}

\newcommand{\op}{\operatorname}

\newcommand{\code}[1]{\lstinline[language=Python]|#1|}

\title[$\lambda$-rings in Python]{Motives meet SymPy: studying $\lambda$-ring expressions in Python}

\author[D. Sanchez]{Daniel Sanchez}
\address{D. Sanchez, 
\newline\indent
Institute for Research in Technology, ICAI School of Engineering, Comillas Pontifical University, Calle del Rey Francisco 4, 28015 Madrid, Spain}
\email{dani.sanchez@alu.comillas.edu}

\author[D. Alfaya]{David Alfaya}
\address{D. Alfaya, 
\newline\indent
Department of Applied Mathematics and Institute for Research in Technology, ICAI School of Engineering, Comillas Pontifical University, C/Alberto Aguilera 25, 28015 Madrid, Spain}
\email{dalfaya@comillas.edu}

\author[J. Pizarroso]{Jaime Pizarroso}
\address{J. Pizarroso, 
\newline\indent
Department of Telematics and Computation and Institute for Research in Technology, ICAI School of Engineering, Comillas Pontifical University, C/Alberto Aguilera 25, 28015 Madrid, Spain}
\email{jpizarroso@comillas.edu}



\begin{document}

\begin{abstract}
We present a new Python package called ``motives'', a symbolic manipulation package based on SymPy capable of handling and simplifying motivic expressions in the Grothendieck ring of Chow motives and other types of $\lambda$-rings. The package is able to manipulate and compare arbitrary expressions in $\lambda$-rings and, in particular, it contains explicit tools for manipulating motives of several types of commonly used moduli schemes and moduli stacks of decorated bundles on curves. We have applied this new tool to advance in the verification of Mozgovoy's conjectural formula for the motive of the moduli space of twisted Higgs bundles, proving that it holds in rank 2 and 3 for any curve of genus up to 18 and any twisting bundle of small degree.
\end{abstract}

\maketitle


\section{Introduction}
The Grothendieck motive of a scheme encodes crucial information about its geometry. For instance, if two schemes share the same class in the Grothendieck ring of Chow motives, then their Hodge E-polynomials and their Poincaré polynomials are the same. For this reason, there have been numerous efforts to compute closed formulas for the motives of several types of moduli spaces and moduli stacks in terms of elementary components \cite{Ba02, BD07,San14, GPHS14, Moz12, Lee18, Gon18, Gon18, GL20, FNZ21, AO24}. However, handling expressions in the Grothendieck ring of motives can be a challenging task, and even comparing if two expressions are equal can become a hard problem. For instance, some of these expressions contain motivic sums with increasing numbers of terms \cite{AO24,GPHS14}, or nontrivial plethystic operations on motivic expressions \cite{Moz12, FNZ21}, whose manipulation can become nontrivial at times.

In order to tackle this problem, in \cite{Alf22} a theoretical algorithm was proposed for simplifying these types of expressions into polynomials in certain sets of motivic generators, as well as expressions in other types of $\lambda$-rings. In that paper, a focus was put in simplifying expressions in the subring of the Grothendieck ring of Chow motives spanned by a finite set of complex algebraic curves and an ad-hoc MATLAB code was implemented to apply this theoretical algorithm to successfully compare two different formulas for the moduli space of $L$-twisted Higgs bundles: a conjectural equation obtained by Mozgovoy as a solution to the ADHM equations \cite{Moz12} and two formulas proven in \cite{AO24} for the Grothendieck virtual class of that moduli space in rank 2 and 3 obtained through a Bialynicki-Birula decomposition of the variety. The resulting program was able to verify that Mozgovoy's conjecture for the motive holds for curves of genus at most 11 and any twisting bundle $L$ of low degree.

In this paper, we have further improved the algorithm proposed in \cite{Alf22} and we have built a full general purpose $\lambda$-ring manipulation Python package integrating the algorithm into the SymPy symbolic framework \cite{sympy_2017}, a Python library for symbolic mathematics. This new library called \texttt{motives} defines an abstract $\lambda$-ring expression class and provides general simplification algorithms capable of simplifying effectively any $\lambda$-ring expression into a polynomial in a finite set of motivic generators. By building upon \texttt{SymPy}, these algorithms further integrate into the symbolic manipulation toolbox from \texttt{SymPy}, extending the functionality of its general purpose simplification and manipulation functionality to work on $\lambda$-rings and, at the same time, enriching it with new tools explicitly aimed to simplify and compare $\lambda$-ring expressions.

The simplification algorithms are designed to work with any user-defined $\lambda$-ring class, but the library also includes some out-of-the-box types of $\lambda$-rings and modules aimed for the working mathematician.

\newpage
The current release of the package includes:
\begin{itemize}
\item $\lambda$-rings of integers
\item Free $\lambda$-rings and free $\lambda$-ring extensions of a $\lambda$-ring
\item Polynomial $\lambda$-ring extensions of a $\lambda$-ring.
\item Grothendieck ring of Chow motives, including the following pre-programmed motives:
\begin{itemize}
\item Complex algebraic curves
\item Jacobian varieties of curves
\item Symmetric and alternated products of any variety given its motive
\item Moduli spaces of vector bundles on curves
\item Moduli spaces of $L$-twisted Higgs bundles on curves
\item Moduli spaces of chain bundles and variations of Hodge structure on curves in low rank
\item Algebraic groups
\item Moduli stacks of vector bundles and principal $G$-bundles on curves
\item Classifying stacks $BG$ for several groups $G$
\end{itemize}
\end{itemize}
and we plan to expand the library with more $\lambda$-rings and geometric elements in future releases. The code of the library is publicly hosted at
\href{https://github.com/CIAMOD/motives}{\texttt{https://github.com/CIAMOD/motives}}. Additionally, \texttt{motives} package can be installed directly from \href{https://pypi.org/}{PyPI} by running \lstinline[language=bash]{pip install motives} in a terminal.

In order to test the capabilities of this -- now general purpose -- symbolic package, we have applied it to the same problem treated originally in \cite{Alf22}: the proof of Mozgovoy's conjectural formula of the moduli space of $L$-twisted Higgs bundles. The new package was proven to be significantly more efficient than the ad-hoc code implemented for the problem in \cite{Alf22}. Using the new algorithm, we were able to prove that Mozgovoy's conjecture holds for curves of genus up to 18 (increasing significantly the current verification, limited to genus 11 curves). Concretely, we prove the following (see Section \ref{section:Mozgovoy} and Theorem \ref{thm:main} for details).

\begin{theorem}[Theorem \ref{thm:main}]
\label{thm:intro}
Mozgovoy's conjectural formula for the motive of the moduli spaces of $L$-twisted Higgs bundles holds in the Grothendieck ring of Chow motives in rank at most 3 for any smooth complex projective curve of genus $g$ such that $2\le g \le 18$ and any line bundle $L$ on the curve such that $2g-1\le \deg(L) \le 2g+18$.
\end{theorem}

To put the increment in the genus bound in perspective, the new obtained motivic polynomials are 5 times larger than the largest motive that the algorithm in \cite{Alf22} was capable of processing before reaching the memory limit of the machine used for the test (128GB), and the new library is able to perform these simplifications up to an order of magnitude faster than its task-specific predecessor, showcasing an enhanced scaling behavior.

The paper is structured as follows. Section \ref{section:lambdaRings} includes some generalities on $\lambda$-rings and introduces some explicit algebraic relations between operators in $\lambda$-rings which will be used by the simplification algorithm. Section \ref{section:Groth} explores the properties of the $\lambda$-rings of Grothendieck motives (Grothendieck ring of varieties, Grothendieck rings of Chow motives and other related rings). The description of the main simplification algorithms and the implementation details are described in Section \ref{section:algorithm}. Section \ref{section:Mozgovoy} explores the application of the \code{motives} library to proving Mozgovoy's conjecture in low genus and the comparison between the results of this library and the ones obtained in \cite{Alf22}. Finally, as the development of the package continues, some future lines of work and expected future additions to the package are described in Section \ref{section:futureWork}.

\noindent\textbf{Acknowledgments.}
This research was supported by project CIAMOD (Applications of computational methods and artificial intelligence to the study of moduli spaces, project PP2023\_9) funded by Convocatoria de Financiaci\'on de Proyectos de Investigaci\'on Propios 2023, Universidad Pontificia Comillas, and by grants PID2022-142024NB-I00 and RED2022-134463-T funded by MCIN/AEI/10.13039/501100011033.

\section{\texorpdfstring{$\lambda$}{lambda}-rings}
\label{section:lambdaRings}

Let $R$ be a unital abelian ring. A \emph{$\lambda$-ring structure on $R$} is a set of maps $\lambda=\{\lambda^n\, : \, R\longrightarrow R\}_{n\in \mathbb{N}}$ such that for each $x,y\in R$ and each $n\in \mathbb{N}$ the following hold.
\begin{itemize}
\item[1)] $\lambda^0(x)=1$
\item[2)] $\lambda^1(x)=x$,
\item[3)] $\lambda^n(x+y)=\sum_{i=0}^n \lambda^i(x)\lambda^{n-i}(x)$
\end{itemize}
The pair $(R,\lambda)$ is called a $\lambda$-ring. A $\lambda$-ring structure is called \emph{special} if, moreover, for each $x,y\in R$ and each $n,m\in \mathbb{N}$ we have
\begin{itemize}
\item[4)] $\lambda^n(xy)=P_n(\lambda^1(x),\ldots,\lambda^n(x),\lambda^1(y),\ldots,\lambda^n(y))$,
\item[5)] $\lambda^n(\lambda^m(x))=P_{n,m}(\lambda^1(x),\ldots,\lambda^{nm}(x))$,
\end{itemize}
where $P_n$ and $P_{n,m}$ are the Grothendieck universal polynomials, defined as follows (for more information, see \cite{Knut73}, \cite{Gri19}). If
$$s_n(\bar{X})=s_n(X_1,\ldots,X_m)=\sum_{1\le i_1<\ldots<i_n\le m} \prod_{k=1}^n X_{i_k}$$
is the elementary symmetric polynomial of degree $d$ in the variables given by $\bar{X}=(X_1,\ldots,X_n)$, then $P_n$ and $P_{n,m}$ are the unique integral polynomials such that
$$P_n(s_1(\bar{X}),\ldots, s_n(\bar{X}),s_1(\bar{Y}),\ldots,s_n(\bar{Y}))=\coeff_{t^n} \prod_{i,j=1}^n(1+tX_iY_j),$$
$$P_{n,m}(s_1(\bar{Z}),\ldots, s_{nm}(\bar{Z}))=\coeff_{t^n} \sum_{\begin{array}{c}I\subset \{1,\ldots,mn\}\\ |I|=m\end{array}} (1+t\prod_{i\in I} Z_i),$$
where $\bar{X}=(X_1,\ldots,X_n)$, $\bar{Y}=(Y_1,\ldots,Y_n)$ and $\bar{Z}=(Z_1,\ldots,Z_{nm})$. Given an element $x\in R$, let
$$\lambda_t(x)=\sum_{n\ge 0} \lambda^n(x) t^n \, \in\, 1+tR[[t]]$$
denote the generating series of the $\lambda$-powers of $x$ in the variable $t$. Condition (3) of the definition of $\lambda$-ring is equivalent to
\begin{equation}
\label{eq:lambdaSum}
\lambda_t(x+y)=\lambda_t(x)\lambda_t(y).
\end{equation}

Given a $\lambda$-ring structure on $R$, we define its opposite $\lambda$-ring structure $\sigma$ as follows. For each $x\in R$, take
\begin{equation}
\label{eq:oposite}
\sigma_t(x) := \left(\lambda_{-t}(x)\right)^{-1}.
\end{equation}
It is easy to verify that if $\lambda$ is a $\lambda$-ring structure, then $\sigma_t(x)=\sum_{n\ge 0} \sigma^n(x) t^n $ defines another $\lambda$-ring structure $\sigma$ on $R$. Thus, any $\lambda$-ring structure on a ring $R$ defines a triple $(R,\lambda,\sigma)$. In this work, we will assume that at least one of the mutually opposite $\lambda$-ring structures $\lambda$ or $\sigma$ is special and, without loss of generality, we will assume that $\sigma$ always denotes a special structure and that $\lambda$ is a structure which may not be special (as will be described in Section \ref{section:Groth}, this will be the situation in the Grothendieck ring of Chow motives).

If we assume that $\sigma$ is a special $\lambda$ ring structure, we can associate to it a set of Adams operations $\psi^n\, : \, R\longrightarrow R$ for each $n\ge 1$ taking
\begin{equation}
\label{eq:defpsi1}
\psi^n(x)=N_n(\sigma^1(x),\ldots,\sigma^n(x)),
\end{equation}
for each $x\in R$, where $N_n(X_1,\ldots,X_n)$ is the Hirzebruch-Newton polynomial, which is the unique polynomial such that
$$\sum_{i=1}^n X_i^n = N_n(s_1(\bar{X}),\ldots,s_n(\bar{X}))$$
for $\bar{X}=(X_1,\ldots,X_n)$. Furthermore, for each $x\in R$, let
$$\psi_t(x):= \sum_{n\ge 1} \psi^n(x) t^n$$
denote the generating function of $\psi^n(x)$ in $n$ in the variable $t$. Then, it can be shown (see, for instance, \cite{Gri19}) that
$$\psi_t(x)=-t\frac{d}{dt}\log \left(\sigma_{-t}(x)\right).$$
Notice that, as $\sigma_{-t}(x)=(\lambda_t(x))^{-1}$, we have
\begin{equation}
\label{eq:defpsi2}
\psi_t(x)=t\frac{d}{dt}\log \left(\lambda(x)\right).
\end{equation}
By \cite[Theorem 9.2]{Gri19}, if $\sigma$ is special, then the following hold.
\begin{enumerate}
\item For each $n\in \NN$, $\psi^n \, : \, (R,\sigma) \longrightarrow (R,\sigma)$ is a $\lambda$-ring homomorphism.
\item $\psi^1(x)=x$.
\item For each $i,j\in \NN$, $\psi^i \circ \psi^j = \psi^{ij}$.
\end{enumerate}

In \cite{Alf22}, several recurrence formulas were inferred and used for computing certain multivariate polynomials which allowed expressing $\lambda$, $\sigma$ and $\psi$ in terms of each other. An efficient computation and manipulation of these polynomials is fundamental to applying the algorithm described in \cite{Alf22} for simplifying algebraic expressions in $\lambda$-rings. However, some of the expressions found in \cite{Alf22} were based on iterated substitutions of some multivariate polynomials inside other multivariate polynomial expressions. This was highly inefficient and we found that it was one of the computational bottlenecks for the application of cited algorithm (for instance, to the verification of Mozgovoy's conjecture, described in section \ref{section:Mozgovoy}).

In order to solve this issue, we will start this work by presenting some alternative equations providing algebraic relations between $\lambda$, $\sigma$ and $\psi$.

Let $\SP_k(n)$ denote the set of ordered partitions of $n\in \mathbb{N}$ into $k$ ordered positive integers, i.e. if $a = (a_1, a_2,\dots,a_k)\in p_k(n)$, then
$$\left\{\begin{array}{l}
a_1 + a_2 + \dots + a_k = n,\\
a_1\ge a_2\ge \ldots \ge a_k>0,\\
a_i\in \mathbb{Z}\quad  \forall i=1,\ldots,k.
\end{array}\right.$$
We also denote $n_i(a)$ as the number of $i$'s in partition $a$. For instance, $a=(3,3,2,1,1)\in p_5(10)$ satisfies $n_1(a)=2$, $n_2(a)=1$, $n_3(a)=2$.

\begin{proposition}
\label{prop:adams2lambda}
Let $\lambda$ and $\sigma$ be two opposite $\lambda$-ring structures on a ring $R$ with no additive torsion and such that $\sigma$ is special and let $\psi^n$ be the Adams operations associated to $\sigma$. Then
\begin{equation}
\label{eq:a2l}
\lambda^{n}(x) = \sum_{i=0}^{n}\sum_{a=(a_1, \dots, a_i )\in \SP_i(n)} \left( \frac{1}{n_1(a)! n_2(a)! \dots n_n(a)!} \prod_j \frac{\psi^{a_j}(x)}{a_j} \right).
\end{equation}
\end{proposition}

\begin{proof}
We solve for $\lambda_t(x)$ in equation \eqref{eq:defpsi2}.
\begin{equation}
\label{eq:lambexp}
\lambda_t(x) = e^{\int \frac{\psi_t(x)}{t} dt}.
\end{equation}
Now, we expand $\int \frac{\psi_t(x)}{t} dt$:
\begin{equation*}
\int \frac{\psi_t(x)}{t} dt = \int \sum_{n=0}^{\infty} \frac{\psi^n(x)\cdot t^n}{t} dt = \sum_{n=0}^{\infty} \int \psi^n(x) t^{n-1} dt = \sum_{n=0}^{\infty} \frac{1}{n} \psi^n(x) t^n.
\end{equation*}
By using this equality and applying the Taylor series expansion of $e^x$ to the right hand side of equation \eqref{eq:lambexp}, we get:
$$\lambda_t(x) = \sum_{i=0}^{\infty}\frac{(\int \frac{\psi_t(x)}{t} dt)^i}{i!} = \sum_{i=0}^{\infty}\frac{(\sum_{n=0}^{\infty} \frac{1}{n} \psi^n(x) t^n)^i}{i!}$$
We can now expand $(\sum_{n=0}^{\infty} \frac{1}{n} \psi^n(x) t^n)^i$ by using the Multinomial Theorem.
\begin{align*}
\lambda_t(x) &= \sum_{i=0}^{\infty}\frac{\sum_{n=0}^\infty\left(\sum_{a = (a_1, a_2,\dots,a_i)\in \SP_i(n)}\binom{i}{n_1(a) n_2(a) \dots n_n(a)}\frac{1}{a_1}\psi^{a_1}(x)\frac{1}{a_2}\psi^{a_2}(x)\dots\frac{1}{a_i}\psi^{a_i}(x)\right)t^n}{i!}\\
 &= \sum_{n=0}^{\infty}\sum_{i=0}^n\sum_{a=(a_1,\dots,a_i)\in \SP_i(n)}\frac{\binom{i}{n_1(a) n_2(a) \dots n_n(a)}\prod_j\psi^{a_j}(x)}{i! \prod_j a_j}t^n\\
 &= \sum_{n=0}^{\infty}\sum_{i=0}^n\sum_{a=(a_1,\dots,a_i)\in \SP_i(n)}\frac{i!\prod_j\psi^{a_j}(x)}{i! \cdot n_1(a)! n_2(a)! \dots n_n(a)! \prod_j a_j}t^n.    
\end{align*}
By equating the coefficients and cancelling the factorials, we achieve the desired equality.
$$
\lambda^{n}(x) = \sum_{i=0}^{n}\sum_{a=(a_1, \dots, a_i )\in \SP_i(n} \left( \frac{1}{n_1(a)! n_2(a)! \dots n_n(a)!} \prod_j \frac{\psi^{a_j}(x)}{a_j} \right).
$$
\end{proof}

\begin{proposition}
\label{prop:adams2sigma}
Let $\sigma$ be a special $\lambda$-ring structure on a ring $R$ with no additive torsion and let $\psi^n$ be the Adams operations associated to $\sigma$. Then
\begin{equation}
\label{eq:a2s}
\sigma^{n}(x) = \sum_{i=0}^{n}\sum_{a=(a_1, \dots, a_i )\in \SP_i(n)} \left( \frac{(-1)^{i+n}}{n_1(a)! n_2(a)! \dots n_n(a)!} \prod_j \frac{\psi^{a_j}(x)}{a_j} \right).
\end{equation}
\end{proposition}

\begin{proof}
The proof is analogous to the previous lemma, using \eqref{eq:defpsi1} instead of \eqref{eq:defpsi2} and adjusting the signs accordingly.
\end{proof}

\begin{proposition}
\label{prop:lambda2adams}
Let $\lambda$ and $\sigma$ be two opposite $\lambda$-ring structures on a ring $R$ with no additive torsion and such that $\sigma$ is special and let $\psi^n$ be the Adams operations associated to $\sigma$. Then
\begin{equation}
\label{eq:l2a}
\psi^n(x) = \sum_{i=1}^{n}(-1)^{i}\sum_{l=0}^{n-1}\sum_{a=(a_1, \dots, a_i )\in p_{i}(l)}(n-l)\frac{i!}{n_1(a)! n_2(a)! \dots n_l(a)!}\lambda^{n-l}(x)\prod_k\lambda^{a_k}(x).
\end{equation}
\end{proposition}

\begin{proof}
We apply the Taylor series expansion of the natural logarithm to the right hand side of equation \eqref{eq:defpsi2}:
\begin{equation*}
\psi_t(x) = t \cdot \frac{d}{dt} \log(\lambda_t(x))= t \cdot \frac{d}{dt} \sum_{i=1}^{\infty}\frac{1}{i}(\lambda_t(x) - 1)^i(-1)^{i-1}.
\end{equation*}
Next, we expand the derivative.
\begin{align*}
\psi_t(x) &= t \sum_{i=1}^{\infty}\frac{1}{i}i\left(\lambda_t(x) - 1\right)^{i-1}\left(\frac{d}{dt}\lambda_t(x)\right)(-1)^{i-1}\\
&= t \sum_{i=1}^{\infty}(-1)^{i-1}\left(\sum_{n=1}^{\infty}\lambda^n(x)t^n\right)^{i-1}\sum_{n=1}^{\infty}n\lambda^n(x)t^{n-1}.
\end{align*}
Now we use the Multinomial Theorem to expand the sum $(\sum_{n=1}^{\infty}\lambda^n(x)t^n)^{i-1}$.
$$
\psi_t(x) = t \sum_{i=1}^{n}(-1)^{i-1}\left(\sum_{n=0}^{\infty}\left(\sum_{a=(a_1, \dots, a_{i-1} )\in p_{i-1}(n)}\binom{i}{n_1(a) n_2(a) \dots n_l(a)}\prod_k\lambda^{a_k}(x)\right)t^n\right)\sum_{n=0}^{\infty}(n+1)\lambda^{n+1}(x)t^n
$$
By applying the convolution, we obtain
\begin{align*}
\psi_t(x) &= t \sum_{i=1}^{n}(-1)^{i-1}\sum_{n=0}^{\infty}\sum_{l=0}^n(n-l+1)\lambda^{n-l+1}(x)\left(\sum_{a=(a_1, \dots, a_{i-1} )\in p_{i-1}(l)}\binom{i}{n_1(a) n_2(a) \dots n_l(a)}\prod_k\lambda^{a_k}(x)\right)t^n\\
&= \sum_{n=0}^{\infty}\sum_{i=1}^{n}(-1)^{i-1}\sum_{l=0}^{n-1}(n-l)\lambda^{n-l}(x)\left(\sum_{a=(a_1, \dots, a_{i-1} )\in p_{i-1}(l)}\binom{i}{n_1(a) n_2(a) \dots n_l(a)}\prod_k\lambda^{a_k}(x)\right)t^n,
\end{align*}
and, equating the coefficients, we arrive to the desired formula
$$
\psi^n(x) = \sum_{i=1}^{n}(-1)^{i}\sum_{l=0}^{n-1}\sum_{a=(a_1, \dots, a_i )\in p_{i}(l)}(n-l)\frac{i!}{n_1(a)! n_2(a)! \dots n_l(a)!}\lambda^{n-l}(x)\prod_k\lambda^{a_k}(x).
$$
\end{proof}

\begin{proposition}
\label{prop:sigma2adams}
Let $\sigma$ be a special $\lambda$-ring structure on a ring $R$ with no additive torsion and let $\psi^n$ be the Adams operations associated to $\sigma$. Then
\begin{equation}
\label{eq:s2a}
\psi^n(x) = \sum_{i=1}^{n}(-1)^{i+k+1}\sum_{l=0}^{n-1}\sum_{a=(a_1, \dots, a_i )\in p_{i}(l)}(n-l)\frac{i!}{n_1(a)! n_2(a)! \dots n_l(a)!}\sigma^{n-l}(x)\prod_k\sigma^{a_k}(x).
\end{equation}
\end{proposition}

\begin{proof}
The proof is analogous to the previous lemma, using \eqref{eq:defpsi1} instead of \eqref{eq:defpsi2} and adjusting the signs accordingly.
\end{proof}

\begin{proposition}
\label{prop:adamsRecursive}
Let $\lambda$ and $\sigma$ be two opposite $\lambda$-ring structures on a ring $R$ such that $\sigma$ is special and let $\psi^n$ be the Adams operations associated to $\sigma$. Then
\begin{equation*}
\psi^n(x) = \sum_{i=1}^{n}(-1)^{n-i}i\sigma^{n-i}(x)\lambda^i(x).
\end{equation*}
This equation can be converted into a polynomial formula in $\lambda$ or $\sigma$ that represent a specific degree of $\psi$ by substituting the corresponding operator by the following polynomial $P_n^{op}\in \ZZ[X_1,\ldots,X_n]$, which expresses $\lambda$ in terms of $\sigma$ and vice versa
\begin{equation}
\label{eq:pop}
\sigma^n=P_n^{op}(\lambda^1(x),\ldots,\lambda^n(x)), \quad \quad \lambda^n=P_n^{op}(\sigma^1(x),\ldots,\sigma^n(x)), \quad \forall x\in R
\end{equation}
defined recursively as follows.
\begin{align*}
    P_{0}^{op}&=1\\
    P_n^{op}&=\sum_{i=0}^{n-1}P_i^{op}X_{n-i} \quad \forall n\ge 1.
\end{align*}
\end{proposition}
\begin{proof}
As $\lambda$ is a $\lambda$-ring structure, by \eqref{eq:lambdaSum} we have
$$\lambda_t(x + y) = \lambda_t(x)\lambda_t(y).$$
By substituting $x$ and $-x$, we get that $$\lambda_t(0) = \lambda_t(x)\lambda_t(-x),$$ and so
\begin{equation}
\label{eq:lambgen}
\lambda_{t}(-x) = (\lambda_{t}(x))^{-1}.
\end{equation}
Next, we use the derivative of the logarithm on the right hand side of equation \eqref{eq:defpsi2}.
$$
\psi_t(x) = t \cdot \frac{d}{dt} \log(\lambda_t(x)) = t \cdot \frac{\frac{d}{dt} \lambda_t(x)}{\lambda_t(x)} = t \cdot \frac{\sum_{n=1}^{\infty}n\lambda^n(x)t^{n-1}}{\lambda_t(x)} = \lambda_t(x)^{-1}\sum_{n=1}^{\infty}n\lambda^n(x)t^{n}.
$$
Now, we apply equation \eqref{eq:lambgen} to this result, yielding
$$
\psi_t(x) = \lambda_t(-x)\sum_{n=1}^{\infty}n\lambda^n(x)t^{n} = \sum_{n=0}^{\infty}\lambda^n(-x)t^n\sum_{n=1}^{\infty}n\lambda^n(x)t^{n}.
$$
We apply the convolution
$$
\psi_t(x) = \sum_{n=1}^{\infty}\sum_{i=1}^{n}i\lambda^i(x)\lambda^{n-i}(-x)t^n.
$$
and, by equalling the coefficients, we get
$$
\psi^n(x) = \sum_{i=1}^{n}i\lambda^i(x)\lambda^{n-i}(-x).
$$
Now we use the fact that $(-1)^{n}\sigma^{n}(x) = \lambda^{n}(-x)$, which follows from \eqref{eq:lambgen} and the definition of opposite structure to obtain
$$
\psi^n(x) = \sum_{i=1}^{n}(-1)^{n-i}i\sigma^{n-i}(x)\lambda^i(x).
$$
By using the equations from \cite[Proposition 2]{Alf22} to express $\sigma^n(x)$ in terms of $\lambda^n(x)$ (or vice versa), we can expand the right-hand side of this equation to get the desired formula.
\end{proof}
Finally, let us recall the notion of dimension of an element $x\in R$ in a $\lambda$-ring. We say that $x$ is $d$-dimensional for the structure $\lambda$ (respectively, for the structure $\sigma$) if $\lambda_t(x)$ (respectively $\sigma_t(x)$) is a degree $d$-polynomial. This implies that $\lambda^n(x)=0$ or $\sigma^n(x)=0$ for each $n>d$ respectively.

\begin{remark}
\label{rmk:finiteDim}
If $x\in R$ is a $d$-dimensional object for $\lambda$, then $\sigma^n(x)$ and $\psi^n(x)$ can be expressed as polynomials in $d$ variables depending on $\lambda^1(x),\ldots,\lambda^d(x)$ through \eqref{eq:l2a} and \eqref{eq:pop}.
\end{remark}

\section{\texorpdfstring{$\lambda$}{lambda}-ring structures in the Grothendieck ring of Chow motives}
\label{section:Groth}

Let $\VarC$ denote the category of isomorphism classes quasi-projective varieties over $\CC$. The Grothendieck ring of varieties, denoted by $K_0(\VarC)$, is the quotient of the free abelian group generated by $\VarC$ quotiented by the following relation. If $Z$ is a closed subvariety of $X$, then
$$[X]=[X\backslash Z] + [Z],$$
where $[X]$ denotes the class of $X$ in $K(\VarC)$. The product of the ring is induced by taking
$$[X][Y]=[X\times Y]$$
for any pair of varieties $X$ and $Y$.

Similarly, let $\ChowC$ denote the category of Chow motives and let $K_0(\ChowC)$ denote the Grothendieck ring of Chow motives. For both rings, we will denote by $\LL:=[\mathbb{A}^1]$ the Lefschetz object and we will also consider completions of these rings with respect to $\LL$. Concretely, let
$$\hat{K}_0(\VarC):=\left\{\sum_{r\ge 0} [X_r] \LL^{-r} \, \middle | \, [X_r] \in \VarC, \text{ and } \lim_{r\to \infty} (\dim(X_r)-r)=-\infty \right\}$$
denote the completion of $K_0(\VarC)$ and let $\hat{K}_0(\ChowC)$ be the completion of $K_0(\ChowC)$ defined in analogous terms (see \cite{Ma68} or \cite{Ba01} and \cite{Ba02} for more details).

These rings admit two natural opposite $\lambda$-ring structures of geometric nature. Given a variety $X\in \VarC$, let
$$\Sym^n(X):= \frac{\overbrace{X\times X\times \ldots X}^n}{S_n}$$
denote its symmetric product. The map $\Sym^n$ extends to a map $\lambda^n \, : \, \hat{K}_0(\VarC) \longrightarrow \hat{K}_0(\VarC)$ induced by
$$\lambda^n([X])=[\Sym^n(X)].$$
By \cite{LL04}, $\lambda$ defines a $\lambda$-ring structure on $K_0(\VarC)$, but it is not a special $\lambda$-ring structure. Instead, its opposite $\lambda$-ring structure, which we will denote by $\sigma$, is special.

Analogous $\lambda$-ring structures also exist in $\hat{K}_0(\ChowC)$. Given a chow motive $[X]\in \hat{K}_0(\ChowC)$, define $\lambda^n([X])=\Sym^n(X)$ as the image of the map
$$\frac{1}{n!}\sum_{\alpha \in S_n} \alpha \, : \, X^{\otimes n} \longrightarrow X^{\otimes n}$$
and define $\sigma^n([X])=\Alt^n(X)$ as the image of the map
$$\frac{1}{n!}\sum_{\alpha \in S_n} (-1)^{|\alpha|} \alpha \, : \, X^{\otimes n} \longrightarrow X^{\otimes n}.$$
Then, $\lambda$ and $\sigma$ are opposite $\lambda$-ring structures on $\hat{K}_0(\ChowC)$ and $\sigma$ is a special $\lambda$-ring structure (see, for instance \cite{H07}). Moreover, the natural map $\hat{K}_0(\VarC) \longrightarrow \hat{K}_0(\ChowC)$ is a $\lambda$-ring map for these structures and, as a consequence of \cite[Proposition 2.7]{LL03}, the localizations $\hat{K}_0(\VarC)$ and $\hat{K}_0(\ChowC)$ have no additive torsion and, therefore, the identities from Section \ref{section:lambdaRings} relating the $\lambda$ structures $\lambda$ and $\sigma$ with the Adams operations $\psi$ in these rings hold.

Given a variety $X$, it is common to use the notation 
$$Z_X(t):= \sum_{n\ge 0} [\Sym^n(X)] t^n = \lambda_t(X).$$
This is usually called the motivic zeta-function of $X$ \cite{H07}.

\section{The \texttt{motives} package}
\label{section:package}

\subsection{The simplification algorithm}\leavevmode\newline
\label{section:algorithm}
The main strategy followed for the simplification algorithm is based on \cite[Algorithm 1 and Theorem 4]{Alf22}, but some further abstractions and simplifications have been made in order to yield a general implementation of such algorithm. In order to simplify arbitrary algebraic expressions possibly combining $\lambda$, $\sigma$ and $\psi$ in non-trivial ways we use the fact that Adams operations are ring homomorphisms (see, for instance, \cite[Theorem 9.2]{Gri19}) and that we can use the expressions obtained in Section \ref{section:lambdaRings} and in \cite[Section 2]{Alf22} to rewrite any other operators $\lambda^n$ or $\sigma^n$ as polynomials in Adams operations. More concretely, as $\psi^n$ is a ring homomorphism for any $n$, if $P(\psi^{k_1}(x_1),\ldots,\psi^{k_s}(x_s))\in \mathbb{Q}[\psi^{k_1}(x_1),\ldots,\psi^{k_s}(x_s)]$ is a rational polynomial expression depending on Adams operations of some $x_1,\ldots,x_s\in R$, then,
\begin{equation}
\label{eq:psi_polynomial}
\psi^n(P(\psi^{k_1}(x_1),\ldots,\psi^{k_s}(x_s)))=P(\psi^{nk_1}(x_1),\ldots,\psi^{nk_s}(x_s)).
\end{equation}
This could also be done if $P$ is a rational expression instead of a polynomial. We can use this principle recursively to transform any expression tree formed by a composition of ring operations and Adams operations into a polynomial (respectively, rational expression) in a finite set of Adams operations of the leaves of the given expression tree.

Furthermore, the equations \eqref{eq:a2l} and \eqref{eq:a2s} provide a way to compute $\lambda^n(x)$ and $\sigma^n(x)$ as polynomials depending on the Adams operations $\psi^k(x)$ for $k\le n$ for any $x\in R$. The \texttt{motives} package combines these two methods to recursively transform any $\lambda$-ring expression into a polynomial in Adams operations as follows.

The package extends the functionality of \texttt{SymPy}'s expression trees to organize and manipulate the $\lambda$-ring formulas, defining a \code{LambdaRingExpr} class which extends \texttt{SymPy}'s \code{Expr} generic symbolic expression class. A function \code{to_adams} has been added to transform any expression tree into a \texttt{SymPy} polynomial depending on the Adams operations of its leaves. A class \code{Operand} has been defined to characterize admissible ``leaf objects" or ``operands" for $\lambda$-ring expressions. In order for the simplification algorithm to take into account possible non-trivial and domain-specific algebraic relations between different Adams operations of a certain type of objects, the computation and symbolic representation of Adams operations of an element is always delegated to the corresponding \code{Operand} object in the following sense.

Any \code{Operand} object representing an element $x\in R$ in the $\lambda$-ring must implement an \code{_apply_adams} method capable of receiving a positive integer $n$ and a polynomial $P$ of the form
\begin{equation}
    \label{eq:to_adams_polynomial}
    P(\psi^{1}(x),\ldots,\psi^{k}(x),\psi^{k_1}(x_1),\ldots,\psi^{k_s}(x_s))
\end{equation}
depending on Adams operations of $x$ and possibly other elements $x_1,\ldots,x_s$ and performs a partial Adams operation on the polynomial, by replacing the instances of the Adams operations of $x$ on which $P$ depends, namely $\psi^1(x),\ldots, \psi^k(x)$, with the result of applying $\psi^n$ to them, leaving any other Adams operation in the expression $P$ unaffected. Concretely, if $P$ is a polynomial as in \eqref{eq:to_adams_polynomial}, then \code{x._apply_adams(n,P)} would yield
$$P(\psi^{n}(x),\ldots,\psi^{nk}(x),\psi^{k_1}(x_1),\ldots,\psi^{k_s}(x_s)).$$
Depending on the specific context of what $x$ represents, the concrete $\lambda$-ring where this expression lives and possible algebraic relations existing between the Adams operations of $x$, the computed elements $\psi^{kn}(x)$ substituted into $P$ might be new algebraically independent symbolic variables created to represent such Adams operations, some expressions (depending possibly on other Adams operations of $x$ or other elements in the ring), etc. Some more complex manipulations of the polynomials could also be possible. This is left to depend on the implementation of the specific \code{Operand}.

Notice that if $P=P(\psi^{k_1}(x_1),\ldots,\psi^{k_s}(x_s))$ is an Adams polynomial depending on $x_1,\ldots,x_s\in R$, then applying \code{P=}$x_i$\code{._apply_adams(n,P)} iteratively for each $i=1,\ldots,s$ would yield exactly $\psi^n(P)$, by \eqref{eq:psi_polynomial}. This behaviour is actually what defines functionally the \code{_apply_adams} method. Each \code{Operand} must implement it so that an iterated application of the method for each element on which an Adams polynomial depends results in the application of $\psi^n$ to such polynomial.

Counting on the existence of this delegated \code{_apply_adams} method on the leaves of an expression tree, the simplification \code{to_adams} algorithm performs a recursive walk through the tree acting as follows on each traversed node.
\begin{enumerate}
    \item If the node is a built-in \texttt{SymPy} operator node corresponding to a ring operation (\texttt{Add}, \texttt{Mul}, \texttt{Pow}), then perform \code{to_adams} at the operand children of the node to transform the operand expressions into Adams polynomials and perform the operation designated by the node on the corresponding results.
    \item If the node is an \texttt{Adams} operation, representing $\psi^n(T)$ for some child expression tree $T$, then apply \code{to_adams} to the child $T$ to obtain a polynomial (respectively, rational expression)
    $$P(\psi^{k_{1,1}}(x_1),\ldots,\psi^{k_{1,n_1}}(x_1),\ldots,\psi^{k_{s,n_s}}(x_s))$$
    which represents $T$ and depends on the Adams operations of the \code{Operand} objects $x_1,\ldots,x_n$ appearing in the leaves of the tree $T$. Then, as described before, iteratively apply \code{P=}$x_i$\code{.to_adams(n,P)} in order to compute $\psi^n(P)=\psi^n(T)$, delegating the computation of the Adams operations to the operands $x_i$ in $T$.
    \item If the node is a \code{Lambda_} or \code{Sigma} operator, i.e. it represents $\lambda^n(T)$ or $\sigma^n(T)$ respectively for some child expression tree $T$, then apply \code{to_adams} to $T$ in order to obtain an equivalent Adams polynomial $P$ to $T$. Then, compute $\psi^k(P)$ for each $k=1,\ldots,n$ by applying the same technique as for the \texttt{Adams} nodes as follows. For each $k$, apply iteratively $x_i$.\code{_apply_lambda(k,P)} for each leaf $x_i$ in $T$ to obtain a polynomial $P_k:=\psi^k(P)$. Then, use the polynomials $L_n$ and $L_n^{op}$ from \cite[Proposition 3 and Corollary 1]{Alf22}, computed through equations \eqref{eq:a2l} and \eqref{eq:a2s} respectively, to compute
    $$\lambda^n(T)=\lambda^n(P)=L_n^{op}(\psi^1(P),\ldots,\psi^n(P))=L_n^{op}(P_1,\ldots,P_k)$$
    or
    $$\sigma^n(T)=\sigma^n(P)=L_n(\psi^1(P),\ldots,\psi^n(P))=L_n(P_1,\ldots,P_k)$$
    as necessary.
    \item Finally, for terminal \code{Operand} nodes (that is, a leaf of the tree), the implementation of the \code{to_adams} method depends on the specific $\lambda$-ring and element being computed. For example, an integer or free symbolic variable will simply return itself.
\end{enumerate}

An additional \code{to_lambda} method was also implemented for transforming the expression into a polynomial depending only on $\lambda$-powers instead of depending on $\lambda$-operations. This becomes useful when the expression depends on objects which are known to be finite-dimensional for the $\lambda$ structure, so that the resulting simplified polynomial only depends on a known finite set of generators (see Remark \ref{rmk:finiteDim}). Moreover, as explained in Section \ref{section:Groth}, in the Grothendieck rings of varieties and Chow motives $\lambda$-operations have an actual geometric meaning and a decomposition in terms of polynomials of such objects can sometimes be interpreted as very useful geometric decompositions of a variety. To compute \code{to_lambda}, first \code{to_adams} is computed, so that the tree $T$ is transformed into a polynomial of the form $T=P(\psi^{k_1}(x_1),\ldots,\psi^{k_s}(x_s))$. Then Proposition \ref{prop:lambda2adams} and Proposition \ref{prop:adamsRecursive} are used to compute $\psi^{k_i}(x_i)$ as a polynomial in terms of $\lambda^1(x_i),\ldots, \lambda^{k_i}(x_i)$ for each $i$. These polynomials are then substituted into $P$ to yield the desired polynomial expression for the tree $T$. This substitution is again delegated to the leaf operands, which must implement a method \code{_subs_adams}. Analogously to \code{_apply_adams}, the \code{_subs_adams} method is called iteratively on the Adams polynomial for each leaf \code{Operand} appearing in the expression Tree, and calling it for a leaf substitutes all instances of Adams operations of the given operand by their corresponding polynomial expressions in terms of its $\lambda$-powers.

One could remark that, if the starting  expression only depends on $\lambda$-operations to begin with, then transforming them completely to Adams operations $\psi^k$ in order to transform them back to $\lambda$-operations $\lambda^k$ at the end could be unnecessary for simplifying it in some situations. For instance, if an expression has some parts which are already polynomials in $\lambda$, then it would be very inefficient to transform that part of the tree to Adams just to transform it back to $\lambda$. This would actually only be needed if there is another operator ($\lambda$, $\sigma$ or $\psi$) acting on the subtree which would require transforming all operators to Adams in order to compute the corresponding composition. In order to avoid this, an optimization was made in the code to detect these types of situations and only transform $\lambda$ operators into Adams when needed, implemented through a \code{_to_adams_lambda} method.

As the polynomials relating $\lambda$, $\sigma$ and $\psi$ are universal, a singleton class \code{LambdaRingContext} was implemented to compute these polynomials and save them cached for reuse to improve efficiency.
\subsection{Core classes in the package}
\subsubsection{Integration with \code{SymPy}} \leavevmode\newline
As mentioned before, the \code{motives} package is built on top of \code{SymPy} \cite{sympy_2017}. By leveraging \code{SymPy}'s existing capabilities for symbolic computation, the \code{motives} package extends its functionality to support operations using $\lambda$-rings and Grothendieck motives. This integration allows for seamless manipulation of $\lambda$-ring expressions using familiar symbolic expressions while introducing specialized classes and methods to handle the motivic expressions. Let us give some details on this integration.

The \code{motives} package introduces a new class, \code{LambdaRingExpr}, to define an abstract $\lambda-$ring expression which inherits and expands the functionality of \code{SymPy}'s \code{Expr}. This class serves as the base for all expressions in the $\lambda$-ring context. Subclasses of \code{LambdaRingExpr} implement methods to handle specific types of expressions, ensuring that they conform to the operations and properties of $\lambda$-rings.

The other base classes that have been extended are the basic \code{SymPy} operators (\code{Add} for addition, \code{Mul} for element-wise multiplication and \code{Pow} for power) and the \code{Rational} class, which represents all rational numbers. For these classes, no new specific extension has been created, but instead they implement specific methods that are needed to manipulate $\lambda$-rings and motives. For instance, they all need the \code{to_adams} method, which converts an arbitrary expression into a polynomial of Adams operators and operands. For rational numbers, an extension of the following natural $\lambda$-ring structures on the ingegers has been implemented:
    $$\lambda^n(x)=\binom{x+n-1}{n}, \quad \quad \sigma^n(x)=\binom{x}{n}, \quad \quad \psi^n(x)=x, \quad \forall x\in \mathbb{Z}.$$

The bulk of the module is on the operands. All of them inherit either from \code{SymPy}'s \code{Symbol}, or from \code{SymPy}'s \code{AtomicExpr}, which is an abstraction of \code{Symbol}. This gives them many useful properties, such as the associative and commutative properties, but it also allows the expression to handle more powerful methods, such as expand or simplify. The $\lambda$-ring operators (\code{Adams}, \code{Lambda} and \code{Sigma}) are extensions of \code{SymPy}'s \code{Function}.

\subsubsection{Base objects}
The core of the package consists on the following objects:
\begin{itemize}
    \item \code{LambdaRingExpr}:
    This is the abstract class from which all of the implemented classes inherit. It sets the layout that all subclasses should follow. It expands the functionality from \code{SymPy}'s \code{Expr} in order to be able to manage $\lambda-$rings.
    \item \code{Operand}:
    This is the abstract class from which all of the implemented operands inherit. Operands are all classes that act as leaf nodes in the expression tree. It specifies all of the necessary methods operands should implement, as well as defining default behaviours for some of them. These necessary methods are \code{get_adams_var}, \code{get_lambda_var}, \code{_to_adams}, \code{_apply_adams}, \code{_to_adams_lambda} and \code{_subs_adams}.
    \item \code{Operators}:
    All nodes in an expression tree that are not operands are operators, i.e. all non root nodes. These include the ring operators \code{Sigma} ($\sigma$), \code{Lambda_} ($\lambda$), and \code{Adams} ($\psi$). The ring operators are implemented as their own classes, inheriting from \code{RingOperator}, which itself extends \code{LambdaRingExpr}. The \code{SymPy} operators, such as \code{Add}, \code{Mul}, and \code{Pow}, are not implemented as subclasses. Instead, the necessary methods are added directly to their classes by defining them as external functions and associating them to the appropriate \code{SymPy} methods. This approach ensures compatibility with \code{SymPy}'s internal instantiation mechanisms, which directly create these classes and would otherwise lack the necessary methods.
    \item \code{LambdaRingContext}: Singleton class which computes and caches the universal polynomials relating $\lambda$, $\sigma$ and $\psi$ operators of a $\lambda$-ring, computed through the equations described in Propositions \ref{prop:adams2lambda}, \ref{prop:adams2sigma}, \ref{prop:lambda2adams}, \ref{prop:sigma2adams}, \ref{prop:adamsRecursive} and in \cite[\S2]{Alf22}. The polynomials are computed on demand either recursively or using the explicit formulas described in Section \ref{section:lambdaRings}, and saved for future use each time a new polynomial is generated.
    \item \code{Object1Dim}: Denotes a $1$-dimensional object in a $\lambda$-ring for the structure $\sigma$. Recall that if an element $x$ in a $\lambda$-ring $(R,\lambda,\sigma)$ is $d$-dimensional (for the $\lambda$-structure $\sigma$) then $\sigma^n(x)=0$ for each $n>d$. Thus, a $1$-dimensional object $x$ satisfies
    $$\sigma_t(x)=1+x$$
    and, as a consequence of equations \eqref{eq:oposite} and \eqref{eq:defpsi2}, we have
    $$\lambda^n(x)=\psi^n(x)=x^n.$$
\end{itemize}

These classes are not expected to be used directly, but instead they serve as base classes for the rest of the objects implemented in the package. 

\subsection{Structure of the main \texorpdfstring{$\lambda$}{lambda}-rings implemented in the \texttt{motives} package}
Let us describe the main classes implemented in \code{motives}
\subsubsection{Base \texorpdfstring{$\lambda$}{lambda}-ring expressions}
\begin{itemize}
    \item \code{Polynomial1Var}: Abstract polynomial symbolic variable yielding a polynomial extension of a $\lambda$-ring. Given a $\lambda$-ring $(R,\lambda)$, the ring of polynomials $R[T]$ acquires a natural $\lambda$-ring structure in which $\lambda^n(T)=T^n$ (in particular, it becomes a $1$-dimensional object for the opposite $\lambda$-ring structure to $\lambda$). It allows defining expressions in polynomial extensions of rings of motives, like the Grothendieck ring of motives.
    \item \code{Free}: Element in a free $\lambda$-ring. If no further information is provided, all its $\lambda$-powers will be treated as independent algebraic elements in the ring.
\end{itemize}

\subsection{Grothendieck motives}
Subpackage for handling expressions in extensions of the Grothendieck ring of Chow motives $\hat{K}_0(\ChowC)$. It contains some classes for handling basic motives, as well as some sub-packages for handling motives depending on algebraic curves, algebraic groups and moduli schemes and stacks.
\begin{itemize}
    \item \code{Motive}: Base class for all Grothendieck motives. All other classes inherit from it. It implements methods for computing the symmetric and alternated powers of a motive.
    \item \code{Lefschetz}: Lefschetz object $\LL=[\mathbb{A}^1]$.
    \item \code{Point}: Class of a point.
    \item \code{Proj}: Class of the projective space $[\mathbb{P}^n]$.
\end{itemize}

\subsubsection{Curves}
Submodule with classes for handling and simplifying motives depending on the class of an abstract smooth complex algebraic curve of genus $g$, for a given $g$. Every smooth complex projective curve $X$ admits a canonical decomposition for its Chow motive
$$[X]=h^0(X)+ h^1(X)+h^2(X)=1+h^1(X)+\LL.$$
and, by \cite{Kapranov00}, $h^1(X)$ is $2g$-dimensional for $\lambda$. Concretely, in the following formulas we will use the notation
$$P_X(t)=Z_{h^1(X)}(t)=\sum_{n=0}^{2g} \lambda^n(h^1(X))t^n$$
\begin{itemize}
    \item \code{Curve}: Abstract complex algebraic curve of genus $g$. It implements the identities and decompositions from \cite{Kapranov00} and \cite{H07} articulating expressions depending on the curve $X$ as algebraic expressions in a finite set of algebraic generators depending on its Chow decomposition $h^1(X)$ (see below for details).
    \item \code{CurveChow}: Chow decomposition $h^1(X)$ of a generic genus $g$ curve $X$. It corresponds to $h^1(X)=[X]-1-\LL$. Since $X$ is considered as an abstract algebraic curve, the only algebraic relations between the elements $\lambda^k(h^1(X))$ being considered during simplification are the following identities arising from \cite{Kapranov00} (see also \cite{H07}).
    $$\lambda^k(h^1(X))=\LL^{k-g}\lambda^{2g-k}(X) \quad \quad \forall g<k\le 2g$$
    $$\lambda^k(h^1(X))=0 \quad \quad \forall k>2g$$
    In particular, by default, $\lambda^k(X)$ for $k=1,\ldots,g$ and $\LL$ are considered as algebraically independent objects in the ring of motives. So are $\lambda^i(X)$ and $\lambda^j(Y)$ for $i,j=1,\ldots,g$ if $X$ and $Y$ are different curves.
    \item \code{Jacobian}: Jacobian of a curve $X$. Computed in terms of the motivic zeta function of $h^1(X)$ (see \cite{Kapranov00} and \cite{H07}).
    $$[\Jac(X)]=\sum_{k=0}^{2g} \lambda^k(h^1(X))=P_X(1)$$
    \item \code{Piccard}: Picard variety of a curve $X$. The motive coincides with that of its Jacobian.
\end{itemize}

\subsubsection{Groups}
Submodule with classes of several complex algebraic groups, based on the formulas from \cite{BD07}. For a connected semisimple complex group $G$, the motive is computed as a polynomial in $\LL$ through the following equation from \cite[Proposition 2.1]{BD07}
\begin{equation}
\label{eq:group}
  [G]=\LL^{\dim G} \prod_{i=1}^r (1-\LL^{-d_i}),
\end{equation}
where $d_i$ are the degrees of the basic invariant generators of $G$ and $r$ is its rank. For classical groups, these were obtained from \cite[Table 1, pp. 59]{Hum90}. The implemented groups include the following.
\begin{itemize}
    \item \code{SemisimpleG}: Class of a general connected semisimple complex algebraic group with exponents $d_1,\ldots,d_n$. Computed through equation \eqref{eq:group}.
    \item \code{A}: Generic group of type $A_n$
    \item \code{B}: Generic group of type $B_n$
    \item \code{C}: Generic group of type $C_n$
    \item \code{D}: Generic group of type $D_n$
    \item \code{E}: Exceptional groups $E_n$, for $n=6,7,8$.
    \item \code{F4}: Exceptional group $F_4$.
    \item \code{G2}: Exceptional group $G_2$.
    \item \code{GL}: $\op{GL}_n(\CC)$
    \item \code{SL}: $\op{SL}_n(\CC)$
    \item \code{PSL}: $\op{PSL}_n(\CC)$    
    \item \code{SO}: $\op{SO}_n(\CC)$
    \item \code{Sp}: $\op{Sp}_{2n}(\CC)$
    \item \code{Spin}: $\op{Spin}_n(\CC)$
\end{itemize}

\subsubsection{Moduli schemes and moduli stacks} 
Submodule with classes for describing the motivic class of some moduli spaces of decorated bundles on curves (in the future, other moduli spaces will be added to this package as well, see Section \ref{section:futureWork}). See Section \ref{section:Mozgovoy} for further details on the considered moduli. It is composed by a submodule for handling moduli schemes and another submodule for handling moduli stacks.

The moduli schemes already implemented are:
\begin{itemize}
    \item \code{VectorBundleModuli}: Class of the moduli space of semistable vector bundles of rank $r$ and degree $d$ on a smooth complex algebraic curve $X$ of genus $g\ge 2$, computing using the following equations from \cite{GPHS14}, \cite{San14} and \cite[Theorem 4.11]{Ba01}, assuming that $r$ and $d$ are coprime.
    \begin{align*}
        [M(X,2,d)]=&\frac{[\Jac(X)]P_X(\LL)-\LL^g[\Jac(X)]^2}{(\LL-1)(\LL^2-1)}\\
        [M(X,3,d)]=&\frac{[\Jac(X)]}{(\LL-1)(\LL^2-1)^2(\LL^3-1)}\Big(\LL^{3g-1}(1+\LL+\LL^2)[\Jac(X)]^2\\
&-\LL^{2g-1}(1+\LL)^2[\Jac(X)]P_X(\LL)+P_X(\LL)P_X(\LL^2)\Big)\\
        [M(X,r,d)]=&\sum_{s=1}^r\sum_{\begin{array}{c} r_1+\ldots+r_n=r\\ r_i>0\end{array}}(-1)^{s-1} \frac{P_X(1)^s}{(1-\LL)^{s-1}}\prod_{j=1}^s\prod_{i=1}^{r_j-1}Z_X(\LL^i)\\
        &\prod_{j=1}^{s-1}\frac{1}{1-\LL^{r_j+r_{j+1}}}\LL^{\sum_{i<j}r_ir_j(g-1)+\sum_{i=1}^{s-1}(r_i+r_{i+1})\left\{-(r_1+\ldots+r_j)d/n\right\}}
    \end{align*}
    where $\{x\}$ denotes the decimal part of $x\in \mathbb{R}$, i.e., $\{x\}=x-\lfloor x \rfloor$.
    \item \code{VHS}: Class of the moduli space of variations of Hodge structure, or moduli space of chains of a given type. The currently implemented types are $(1,1)$, $(1,2)$, $(2,1)$, $(1,1,1)$ and $(r)$ (which corresponds to moduli spaces of vector bundles). Further types will be added in the future. Computed using the equations from \cite{GPHS14}, \cite{San14} and the computations from \cite{AO24}. See \cite[\S 8]{AO24} for details.
    \item \code{TwistedHiggsModuli} Class of the moduli space of $L$-twisted Higgs bundles of rank $r$ and degree $d$ on a genus $g\ge 2$ smooth complex algebraic curve. Two methods have been implemented for computing this class.
    \begin{itemize}
        \item \code{TwistedHiggsModuliBB}: Class computed for rank 2 and 3 of the moduli space of $L$-twisted through a Bialynicki-Birula of the moduli using \cite[Corollary 8.1]{AO24} and \cite[Theorem 3]{GPHS14}. These formulas were proven when the rank and degree are coprime.
        \item \code{TwistedHiggsModuliADHM}: Conjectural class for the moduli on arbitrary rank and degree. Based on the formulas from \cite[Conjecture 3]{Moz12}, which are solutions to the motivic ADHM equations \cite{CDP11}.
    \end{itemize}
\end{itemize}

The moduli stacks already implemented are:
\begin{itemize}
    \item \code{BG}: Classifying space for the group $G$, $BG=[pt/G]$ for a connected semisimple complex algebraic group $G$. It is computed as
    $$[BG]=1/[G],$$
    with $[G]$ computed through equation \eqref{eq:group}. This formula is conjectural in general, but it has been proven for special groups (like $\op{GL}_n(\CC)$, $\op{SL}_n(\CC)$ and $\op{Sp}_{2n}(\CC)$) \cite[Example 2.6]{BD07}, for $\op{PSL}_n(\CC)$ if $n=2,3$ \cite[Theorem A]{Ber16}, for $\op{SO}_n(\CC)$ \cite[Theorem 3.7 and Corollary 3.8]{DY16}) and for $\op{O}_n$ \cite[Theorem 3.1 and Corollary 3.2]{TV17}.
    \item \code{Bun}: Moduli stack of principal $G$-bundles on a smooth complex projective curve $X$, computed through the following conjectural formula from \cite[Conjecture 3.4]{BD07}
    $$[\mathfrak{Bun}(X,G)]=|\pi_1(G)|\LL^{(g-1)\dim G} \prod_{i=1}^r Z_C(\LL^{-d_i}),$$
    where $d_i$ are the exponents of the group and $r$ its rank. The conjecture was proven for $G=\op{SL}_n(\CC)$ in \cite[\S 6]{BD07}.
\end{itemize}

\subsection{Usage example}

\code{motives} can simplify any expression into a polynomial in Adams or  $\lambda$ operations by calling the methods \code{to_adams()} or \code{to_lambda()} on the expression respectively. An example goes as follows. Calling \code{to_adams()} on the following expression
$$
    \lambda^2\left(\psi^2(x)-\frac{y}{2}\right) 
$$
by using this Python code
\noindent
\begin{lstlisting}
x, y = Free("x"), Free("y")

expr = (x.adams(2) - y / 2).lambda_(2)
print(expr.to_adams())
\end{lstlisting}
\noindent
yields the following equivalent polynomial expression depending solely on $x$, $y$ and the necessary Adams operations of $x$ and $y$.
$$
    -\frac{\psi^2(y)}{4}+\frac{\psi^4(x)}{2}+\frac{(\psi^2(x)-\frac{y}{2})^2}{2}.
$$
In the code, $x$ and $y$ were declared as \code{Free} objects, so no further assumptions were made on them and the computation was performed in the free $\lambda$-ring spanned by $x$ and $y$. If \code{to_lambda()} is called on the expression instead of \code{to_adams()}, we get the following polynomial depending on the $\lambda$ powers of $x$ and $y$ instead of their Adams operations
$$
\frac{-x^4}{2}+2x^2\lambda^2(x)-2x\lambda^3(x)+\frac{y^2}{4}-(\lambda^2(x))^2-\frac{\lambda^2(y)}{2}+2\lambda^4(x)+\frac{(-x^2-\frac{y}{2}+2\lambda^2(x))^2}{2},
$$
which, if we simplify it using the \code{simplify()} method provided by \code{Sympy}:
\noindent
\begin{lstlisting}
print(expr.to_lambda().simplify())
\end{lstlisting}
\noindent
yields the following polynomial still depending on the $\lambda$ powers of $x$ and $y$.
$$
\frac{x^2 y}{2} - 2x \lambda^3(x) + \frac{3y^2}{8} - y \lambda^2(x) + \lambda^2(x)^2 - \frac{\lambda^2(y)}{2} + 2\lambda^4(x).
$$

\section{Applications to the computation of the motive of twisted Higgs bundles}
\label{section:Mozgovoy}

Let $X$ be a smooth complex projective curve of genus $g\ge 2$ and let $L$ be a line bundle on $X$ of degree $\deg(L)=2g-2+p$ with $p>0$. An $L$-twisted Higgs bundle of rank $r$ on $X$ is a pair $(E,\varphi)$ consisting on a rank $r$ vector bundle $E$ and a homomorphism $\varphi\in H^0(\End(E)\otimes L)$. We say that $(E,\varphi)$ is semistable if for any subbundle $F\subsetneq E$ such that $\varphi(F)\subseteq F\otimes L$ the following inequality holds.
$$\frac{\deg(F)}{\rk(F)}\le \frac{\deg(E)}{\rk(E)}.$$

Let $\SM_L(X,r,d)$ denote the moduli space of semistable $L$-twisted Higgs bundles on $X$ of degree $r$ and rank $d$. Through this section, assume that $r$ and $d$ are coprime. There is a natural $\CC^*$ action on the moduli space given by $t\cdot (E,\varphi)=(E,t\varphi)$, giving the moduli space the structure of a smooth semiprojective variety of dimension $1+r^2(2g-2+p)$. In \cite{Moz12}, Mozgovoy stated the following conjectural formula for this moduli space, which is a solution to the motivic ADHM recursion formula \cite{CDP11}.

\begin{conjecture}{ {\cite[Conjecture 3]{Moz12}}}
\label{conj:ADHM}
For each integer $n\ge 1$, let
$$\SH_n(t)=\sum_{\lambda\in \SP(n)} \prod_{s\in d(\lambda)} (-t^{a(s)-l(s)} \LL^{a(s)})^p t^{(1-g)(2l(s)+1)}Z_X(t^{h(s)}\LL^{a(s)}),$$
where $\SP(n)$ denotes the set of ordered partitions of $n$, a partition $\lambda\in \SP(n)$ is considered as a non-increasing sequence of positive integers $\lambda_1\ge \lambda_2 \ge \cdots \ge \lambda_k>0$ summing $n$ and, for each $\lambda\in \SP_n$, 
$$d(\lambda)=\{(i,j)\in \ZZ^2 | 1\le i, \, 1\le j\le \lambda_i\},$$
$$a(i,j)=\lambda_i-j, \quad l(i,j)=\max\{l|\lambda_l\ge j\}-i, \quad \quad h(i,j)=a(i,j)+l(i,j)+1.$$
From $\SH_n(t)$, define $H_r(t)$ for each $r\ge 1$ as follows
\begin{equation*}
\sum_{r\ge 1} H_r(t) T^r =(1-t)(1-\LL t) \sum_{j\ge 1} \sum_{k\ge 1} \frac{(-1)^{k+1}\mu(j)}{jk}  \Bigg ( \sum_{n\ge 1} \psi_j[\SH_n(t)] T^{jn} \Bigg)^k.
\end{equation*}
Then $H_r(t)$ is a polynomial in $t$ and
\begin{equation}
\label{eq:ADHM}
[\SM_L(X,r,d)] =M_{g,r,p}^{\op{ADHM}}:= (-1)^{pr} \LL^{r^2(g-1)+p\frac{r(r+1)}{2}} H_r(1).
\end{equation}
\end{conjecture}

This conjecture implies analogous formulas for the E-polynomial of the moduli space and for the number of rational points in the moduli space, which were later proven in \cite[Theorem 1.1 and Theorem 4.6]{MG19}.

On the other hand, in \cite{AO24}, the following formulas were proven for the virtual classes of moduli spaces of $L$-twisted Higgs bundles of degree $d$ and rank at most $3$ by computing the Bialynicki-Birula decomposition of the variety.

\begin{theorem}[ {\cite[Theorem 3]{GPHS14}, \cite[Corollary 8.1]{AO24}}]
\label{thm:BB}
The following equalities hold in $\hat{K}_0(\VarC)$.

\begin{enumerate}
\item For $r=1$
\begin{equation}
\label{eq:BB1}
[\SM_{L}(X,1,d)]=M^{\op{BB}}_{g,1,p}:=[\Jac(X)\times H^0(X,L^\vee)]=\LL^{g-1+p}P_X(1)
\end{equation}
\item For $r=2$, if $(2,d)=1$,
\begin{equation}
\label{eq:BB2}
\begin{split}
[\SM_{L}(X,2,d)]=M^{\op{BB}}_{g,2,p} &:= \frac{\LL^{4g-4+4p}\Big(P_X(1)P_X(\LL)-\LL^gP_X(1)^2\Big)}{(1-\LL)(1-\LL^2)}\\
& \ \ \ + \LL^{4g-4+3p}P_X(1)\sum_{i=1}^{\lfloor\frac{2g-1+p}{2}\rfloor}\lambda^{2g-1+p-2i}([X]).
\end{split}
\end{equation}
\item For $r=3$, if $(3,d)=1$,
\begin{multline}
\label{eq:BB3}
[\SM_{L}(X,3,d)] = \frac{\LL^{9g-9+9p}P_X(1)}{(\LL-1)(\LL^2-1)^2(\LL^3-1)}\Big(\LL^{3g-1}(1+\LL+\LL^2)P_X(1)^2\\
-\LL^{2g-1}(1+\LL)^2P_X(1)P_X(\LL)+P_X(\LL)P_X(\LL^2)\Big)\\
+\frac{\LL^{9g-9+7p}P_X(1)^2}{\LL-1}\sum_{i=1}^{\lfloor\frac{1}{3}+\frac{2g-2+p}{2}\rfloor}\bigg(\LL^{i+g}\lambda^{-2i+2g-2+p}([X]+\LL^2) -\lambda^{-2i+2g-2+p}([X]\LL+1)\bigg)\\
 +\frac{\LL^{9g-9+7p}P_X(1)^2}{\LL-1}\sum_{i=1}^{\lfloor\frac{2}{3}+\frac{2g-2+p}{2}\rfloor}\bigg(\LL^{i+g-1}\lambda^{-2i+2g-1+p}([X]+\LL^2) -\lambda^{-2i+2g-1+p}([X]\LL+1)\bigg)\\
+\LL^{9g-9+6p}P_X(1)\sum_{i=1}^{2g-2+p} \sum_{j= \max\{2-2g-p+i, 1 -i\}}^{\lfloor (2g-1+p-i)/2\rfloor}\lambda^{-i+j+2g-2+p}([X])\lambda^{-i-2j+2g-1+p}([X]).
\end{multline}
\end{enumerate}
\end{theorem}

A direct comparison between the conjectured expression \eqref{eq:ADHM} and the proved expressions \eqref{eq:BB1}, \eqref{eq:BB2} and \eqref{eq:BB3} is not immediate, specially due to the Adams operations acting on the terms $\SH_n(t)$ in \eqref{eq:ADHM} and the sums in the formulas \eqref{eq:BB2} and \eqref{eq:BB3}. In \cite{Alf22}, a MATLAB code was written for manipulating these precise expressions using the general simplification algorithm described in that algorithm. As a result, it was proven that the conjectural formulas $M_{g,r,p}^{\op{ADHM}}$ from \ref{conj:ADHM} and the proved formulas $M_{g,r,p}^{\op{BB}}$ from Theorem \ref{thm:BB} agree in $\hat{K}_0(\ChowC)$ in the following cases:
\begin{itemize}
\item $X$ is any curve of genus $g$ with $2\le g\le 11$,
\item $L$ is any line bundle on $X$ of degree $2g-1\le \deg(L) \le 2g+18$, and
\item $1\le r \le 3$.
\end{itemize}
The computations were carried in an Intel(R) Xeon(R) E5-2680v4@2.40GHz with 128GB of RAM, and it was found that the main limitation for extending the verification beyond these limits was the over-exponential growth of the memory usage of the MATLAB code with respect to the genus. The test for $r=3$, $g=11$ curves and $\deg(L)=2g+18$ used all the available RAM in the machine. As a way to test the general purpose simplification library proposed in this work, we have revisited this problem using the new \code{motives} package and compared its performance with the previous ad-hoc MATLAB implementation of the algorithm. First, the following simplifications, which were also applied in \cite{Alf22}, were performed in the equation \eqref{eq:BB3}. As $\lambda$ is a $\lambda$-ring structure and $\LL$ is a 1-dimensional object, then for any $n\ge 0$
$$\lambda^n([X]+\LL^2) = \sum_{k=0}^n \lambda^k([X]) \lambda^{n-k}(\LL^2)=\sum_{k=0}^n \lambda^k([X])\LL^{2n-2k},$$
$$\lambda^n([X]\LL+1)=\sum_{k=0}^n \lambda^k([X]\LL) = \sum_{k=0}^n \lambda^k([X])\LL^k.$$
The \code{to_lambda} method from \code{motives} was then applied to both expressions for the motive of the moduli space. In the case of \eqref{eq:ADHM}, terms in $t-1$ were collected and cancelled before evaluating $H_r(1)$, and the resulting polynomials were compared through a standard \texttt{SymPy} comparison of symbolic expressions.

We run an instance of \code{motives} in the same machine with the same RAM limitations for this problem. The library was able to greatly surpass the capabilities of the previous ad-hoc code and it compared both expressions successfully, showing that they are equal in $\hat{K}_0(\ChowC)$ whenever
\begin{itemize}
\item $X$ is any curve of genus $g$ with $2\le g\le 18$,
\item $L$ is any line bundle on $X$ of degree $2g-1\le \deg(L) \le 2g+18$, and
\item $1\le r \le 3$.
\end{itemize}
As a consequence, Theorem \ref{thm:BB} implies that Mozgovoy's conjecture holds under those conditions, yielding the following theorem.
\begin{theorem}
\label{thm:main}
Let $X$ be any smooth complex projective curve of genus $2\le g\le 18$. Let $L$ be any line bundle on $X$ of degree $\deg(L)=2g-2+p$ with $0<p\le 20$. If $r\le 3$ and $d$ is coprime with $r$, then, in $\hat{K}_0(\ChowC)$,
$$[\SM_L(X,r,d)] =  M_{g,r,p}^{\op{BB}}.$$
\end{theorem}

\begin{figure}
    \centering
    \includegraphics[width=0.7\linewidth]{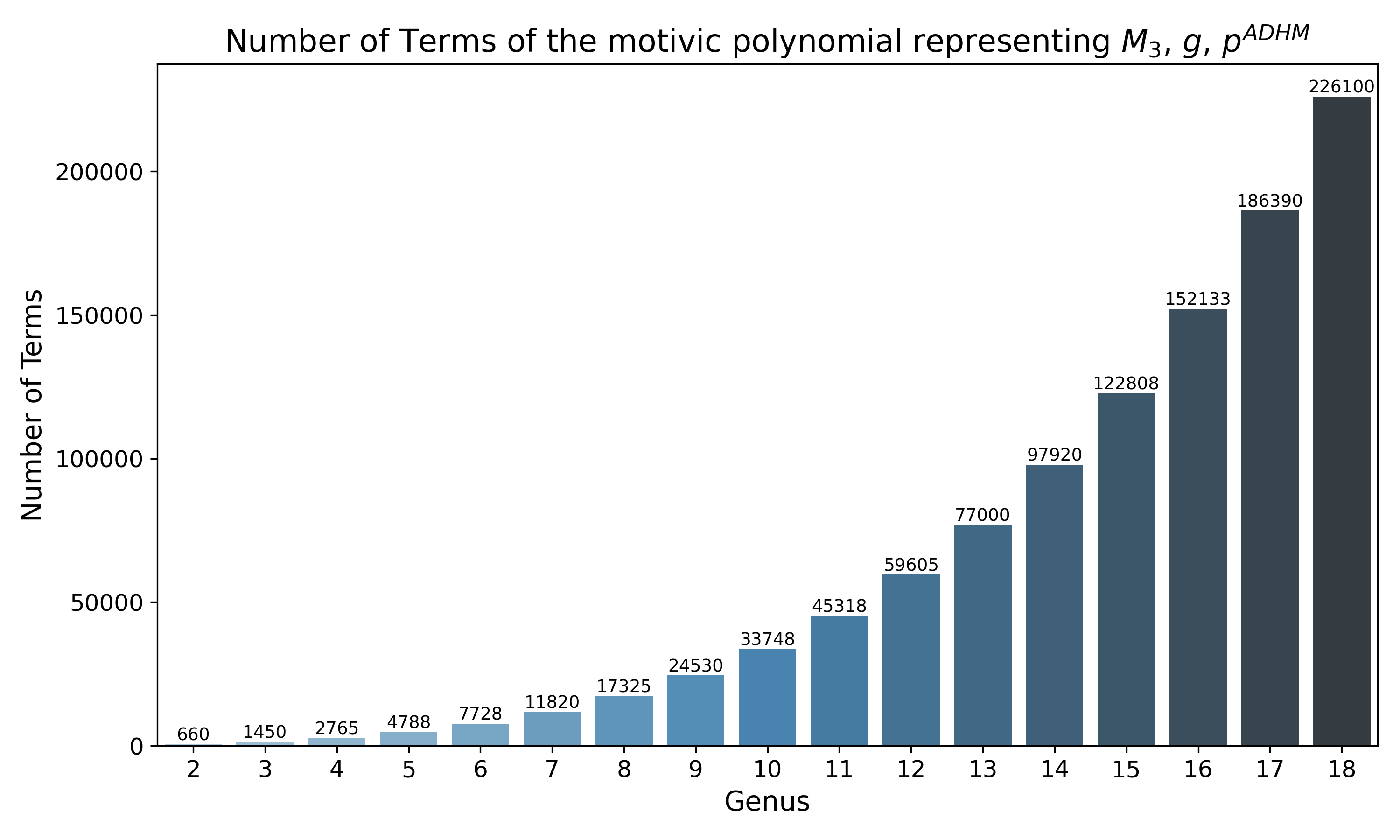}
    \caption{Size comparison between the polynomials generated for $M_{g,3,20}^{\op{BB}}=M_{g,3,20}^{\op{ADHM}}$ after simplification for different $g$, all taken with $r=3$ and $p=20$.}
    \label{fig:polSize}
\end{figure}

To put in perspective the jump from genus $g=11$ to genus $g=18$, one should consider how the expressions grow with respect to the genus. The moduli space has dimension $\dim(\SM_L(X,r,d))=r^2\deg(L)+1$, so the largest moduli space processed in \cite{Alf22}, corresponding to genus $11$, rank $3$ and $\deg(L)=2g+18=40$ had dimension $361$. The largest moduli space considered in the current work, corresponding to genus 18, rank 3 and $\deg(L)=2g+18=54$, has now dimension $487$. The algorithm simplifies the motivic expressions into a multivariate polynomial in $g+1$ (independent) generators and its degree is at most the dimension of the moduli space. In \cite{Alf22}, the maximum test performed resulted in a degree 361 polynomial in 12 variables with 45318 terms. In contrast, the new results were attained for a moduli space whose motive is represented by a degree 487 polynomial in 19 variables with 226100 terms, 5 times larger that the largest expression which the old code was able to handle in that given machine. Figure \ref{fig:polSize} shows the increase in the number of terms of the computed simplified expression. It is important to notice that the computational complexity of manipulating and simplifying the type of nested polynomial expressions arising naturally from the application of the method to the given expressions raises significantly when the number of variables of the polynomials increases, even slightly. Consider, for instance, that even though the obtained final simplified polynomials are relatively small in this case, polynomials of degree 361 in 12 variables can have at most $\sim 3.4\cdot 10^{24}$ possible terms, whereas polynomials of degree 487 in 19 variables could have up to $\sim 3.1 \cdot 10^{36}$ terms. 

Figure \ref{fig:timeComparison} shows a comparison between the time taken by the algorithm in \cite{Alf22} and the \code{motives} library to simplify and compare the motivic expressions $M_{g,r,p}^{BB}$ and $M_{g,r,p}^{ADHM}$ for $r=3$, $p=20$ and different values of $g$. The variation in execution time with respect to other choices of $p$ is relatively small, so only the time needed for computing the biggest moduli for each $g$ ($r=3$ and $p=20$) with the two methods is shown in the figure for clarity. The chart for the \cite{Alf22} algorithm ends at $g=11$ as the memory limit for the machine was reached at that point. During the experiments, running the \code{motives} library for $g=18$ also used a significant part of the memory, but there was still room to continue increasing the genus. The test of the library was stopped at $g=18$ due to the exponential growth of the run time of the algorithm for the given problem.

\begin{figure}
    \centering
    \includegraphics[width=0.7\linewidth]{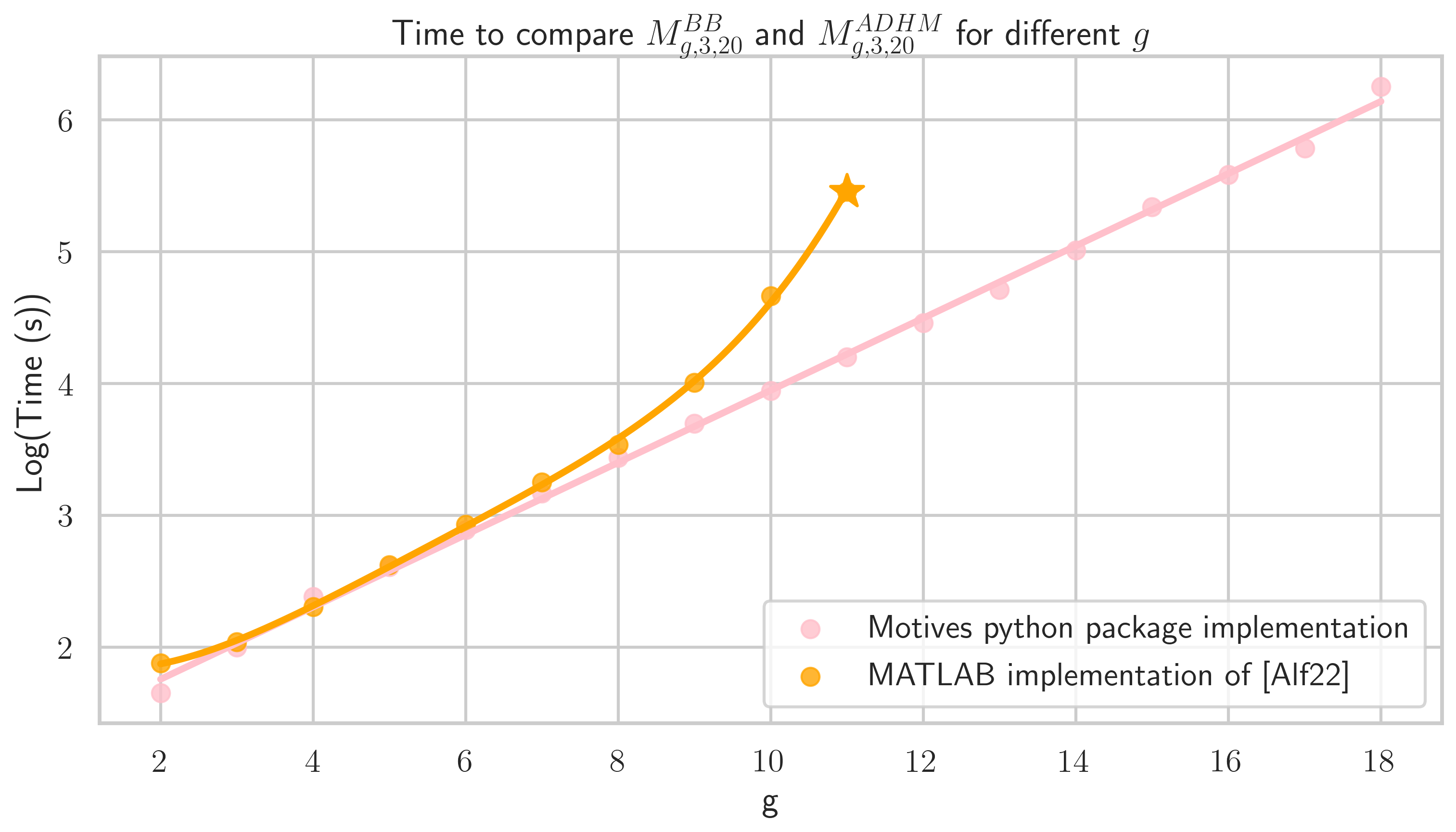}
    \caption{Time comparison between algorithm \cite{Alf22} and \code{motives} package for the computational verification of Mozgovoy's conjectural formula. \cite{Alf22} data is limited to $g\le 11$ because the program reached the memory limit for the machine for that $g$.}
    \label{fig:timeComparison}
\end{figure}

It is worth remarking that the new (general purpose) library took less time to compute the biggest $g=15$ case (whose simplified polynomial has 122808 terms) than it took for the code in \cite{Alf22} to compute the $g=11$ case (with 45318 terms, almost three times less), and it computed the $g=11$ cases an order of magnitude faster than its predecessor ad-hoc implementation for the problem.

Another great benefit of the library is its simplicity of use. For example, the code needed to simplify and check the equality of $M_{g,r,p}^{BB}$ and $M_{g,r,p}^{ADHM}$ is as follows.

\noindent
\begin{minipage}{\dimexpr\linewidth-3\fboxsep-3\fboxrule}
\begin{lstlisting}
cur = Curve("x", g=g)

# Compute the motive of rank r using ADHM derivation
adhm = TwistedHiggsModuli(x=cur, p=p, r=r, method="ADHM")
eq_adhm = adhm.compute(verbose=verbose)

# Compute the motive of rank r using BB derivation
bb = TwistedHiggsModuli(x=cur, p=p, r=r, method="BB")
eq_bb = bb.compute(verbose=verbose)

# Compare the two polynomials
if eq_adhm - eq_bb == 0:
    print("Polynomials are equal")
\end{lstlisting}
\end{minipage}

\section{Future Work}
\label{section:futureWork}
We plan to continue the development of the library \code{motives} expanding its functionalities in different directions and exploring new applications of its symbolic manipulation capabilities to other problems related to motives of moduli spaces.

Some features to be implemented in future releases of \code{motives} include the following.
\begin{itemize}
    \item Tools for computing E-polynomials, Poincaré polynomials and other invariants (like the dimension) from a motive.
    \item Further manipulation tactics capable of partial simplification or manipulation of expression trees in $\lambda$-rings, which complement the full simplification algorithm described in this paper.
    \item Implementations for equations of new types of motives of commonly used geometric constructions, like character or representation varieties, general formulas for moduli spaces of chain bundles and other additional moduli spaces and moduli stacks of bundles and decorated bundles on curves.
    \item Implementations for other commonly used $\lambda$-rings.
\end{itemize}

\bibliographystyle{alpha}
\bibliography{ref}
\end{document}